\documentclass[12pt,a4paper]{amsart}
\usepackage{amssymb,amscd,amsxtra,calc,mathrsfs}
\usepackage[colorlinks=true,citecolor=blue]{hyperref}
\usepackage[all]{xy}
\usepackage{cite}

\usepackage[top=1in, bottom=1in, left=1in, right=1in]{geometry}

\theoremstyle{plain}
    \newtheorem{thm}{Theorem}[section]
    
    \newtheorem{claim}[thm]{Claim}
    \newtheorem{corollary}[thm]{Corollary}
    
    \newtheorem{lemma}[thm]{Lemma}
    \newtheorem{proposition}[thm]{Proposition}

\theoremstyle{definition}

    \newtheorem{remark}[thm]{Remark}
\theoremstyle{remark}

    \newtheorem{setup}[thm]{}

\newcommand{\R}{\mathbb{R}}

\newcommand{\Z}{\mathbb{Z}}

\newcommand{\BK}{\mathcal B\mathcal K}

\newcommand{\OO}{\mathrm{O}}

\newcommand{\cC}{\mathcal{C}}
\newcommand{\cK}{\mathcal{K}}

\newcommand{\alb}{\operatorname{alb}}
\newcommand{\Aut}{\operatorname{Aut}}
\newcommand{\Bir}{\operatorname{Bir}}

\newcommand{\GL}{\operatorname{GL}}

\newcommand{\id}{\operatorname{id}}

\newcommand{\Ker}{\operatorname{Ker}}
\newcommand{\Eff}{\overline{\operatorname{Eff}}}
\newcommand{\Mov}{\overline{\operatorname{Mov}}}

\newcommand{\Nef}{\operatorname{Nef}}
\newcommand{\NS}{\operatorname{NS}}

\newcommand{\Pic}{\operatorname{Pic}}

\newcommand{\Per}{\operatorname{Per}}
\newcommand{\rank}{\operatorname{rank}}

\newcommand{\torsion}{\operatorname{torsion}}

\newcommand{\Alb}{\operatorname{Alb}}

\newcommand{\ratmap}{\dashrightarrow}

\newcommand{\bbR}{\mathbb{R}}
\newcommand{\bbC}{\mathbb{C}}
\newcommand{\bbN}{\mathbb{N}}
\newcommand{\bbQ}{\mathbb{Q}}
\newcommand{\bbZ}{\mathbb{Z}}
\newcommand{\bbP}{\mathbb{P}}

\begin{document}
\title[Criteria for the existence of equivariant fibrations]
{Criteria for the existence of equivariant fibrations on algebraic surfaces
and hyperk\"ahler manifolds
and equality of automorphisms up to powers - a dynamical viewpoint}

\author{Fei Hu}
\address{\textsc{Department of Mathematics} \endgraf \textsc{National University of Singapore, 10 Lower Kent Ridge Road, Singapore 119076}}
\email{hf@u.nus.edu}

\author{JongHae keum}
\address{\textsc{School of Mathematics} \endgraf \textsc{Korea Institute for Advanced Study, Seoul 130-722, Korea}}
\email{jhkeum@kias.re.kr}

\author{De-Qi Zhang}
\address{\textsc{Department of Mathematics} \endgraf \textsc{National University of Singapore, 10 Lower Kent Ridge Road, Singapore 119076}}
\email{matzdq@nus.edu.sg}


\dedicatory{}

\begin{abstract}
Let $X$ be a projective surface or a hyperk\"ahler manifold and $G \le \Aut(X)$.
We give a necessary and sufficient condition for the existence of a non-trivial $G$-equivariant fibration on $X$.
We also show that two automorphisms $g_i$ of positive entropy and polarized by the same nef divisor
are the same up to powers, provided that either $X$ is not an abelian surface or the $g_i$
share at least one common periodic point. The surface case is known among experts,
but we treat this case together with the hyperk\"ahler case using the same language of hyperbolic lattice
and following Ratcliffe \cite{Rat} or Oguiso \cite{Oguiso07}.

This arXiv version contains proofs omitted in the print version.
\end{abstract}



\maketitle

\section{Introduction}
\noindent We work over the field $\bbC$ of complex numbers.
Let $X$ be a smooth projective variety or a compact K\"ahler manifold of dimension $n$.
For a subgroup $G$ of the automorphism group $\Aut(X)$ of $X$ and a $G$-invariant
subgroup $V$ of some cohomology group of $X$, we denote the induced action of $G$ on $V$ by $G|V$.
For example the natural pullback action of $\Aut(X)$ on $H^*(X, \bbC)$
is denoted by $\Aut(X) | H^*(X, \bbC)$.
For an automorphism $g\in \Aut(X)$, let
$$\rho(g) = \rho(g^*) := \max \{|\lambda| : \lambda \text{ is an eigenvalue of } g^*|\bigoplus_{i\ge 0}H^i(X, \bbC) \}$$
be the spectral radius of the pullback action $g^*$ on the total cohomology ring of $X$.
We define the (topological) {\it entropy} as $h(g) = \log \rho(g)$.
By the fundamental work of Gromov and Yomdim, the above definition is equivalent to the original dynamical definition of entropy (cf. \cite{Gromov03}, \cite{Yomdin87}).
The $i$-th {\it dynamical degree} is defined as $d_i(g) := \rho(g^* | H^{i,i}(X, \bbC))$.

An element $g\in \Aut(X)$ is of {\it null entropy} (resp. {\it positive entropy}) if its entropy $h(g) = 0$ (resp. $>0)$.
For a subgroup $G$ of $\Aut(X)$, we define the {\it null-subset} of $G$ as
$$N(G):=\{g\in G : g \text{ is of null entropy, i.e., }h(g)=0\}.$$
In general, $N(G)$ may {\it not} be a subgroup of $G$.
A group $G \le \Aut(X)$ is of null entropy if every $g\in G$ is of null entropy, i.e., $G=N(G)$.

It is known that
$$h(g) = \max_{0 \le i \le n} \log d_i(g)$$
and $h(g) > 0$ if and only if $d_i(g) > 1$ for some (or equivalently for all) $i \in \{1, \dots, n-1\}$; especially $h(g) > 0$ if and only if $d_1(g) > 1$.
When $X$ is projective, $d_1(g) = \rho(g^* | \NS_{\bbC}(X))$,
where $\NS(X)$ is the N\'eron-Severi group of $X$ and $\NS_{\bbC}(X) := \NS(X) \otimes_{\bbZ} \bbC$.
In particular, for an automorphism $g$ of a projective surface $X$, we have $$\rho(g) = d_1(g) = \rho(g^* | \NS_\bbC(X));$$
see \cite{DS04} and the references therein.

By the classification of surfaces, a compact complex surface $S$ has an automorphism of positive entropy only if
$S$ is bimeromorphic to a rational surface, a $K3$ surface, an Enriques surface or a complex torus (cf. \cite{Cantat99}).

Let $X$ be a hyperk\"ahler manifold. Thanks to Beauville, Bogomolov and Fujiki, there exists a natural bilinear (primitive)
integral form $q_X$
of signature $(3, b_2-3)$ on the second integral cohomology group $H^2(X, \bbZ)$.
Moreover, if $X$ is projective, then the restriction of $q_X$ on the N\'eron-Severi group $\NS(X)$ is non-degenerate of signature $(1,\rho(X)-1)$,
where $\rho(X)$ is the Picard number of $X$ (cf. \cite{Huybrechts99}).

Hence algebraic surfaces and projective hyperk\"ahler manifolds are very similar when we focus on their intersection forms: these forms are both represented by hyperbolic lattices.

Below is a hyperk\"ahler analogue of Theorem \ref{Surface equiv} for surfaces.
It uses the deep results due to Bayer - Macr{\`{\i}}, Markman, Matsushita and Yoshioka (cf. \cite{BM13}, \cite{Markman13}, \cite{Matsushita13} and \cite{Yoshioka12}).
Recall that a hyperk\"ahler manifold is of
type $K3^{[n]}$ (resp. generalized Kummer) if it is deformation equivalent
to $S^{[n]}$ for some $K3$ surface $S$ (resp. to a generalized Kummer variety); see details from the above references.
Since a bimeromorphic map $g$ on a hyperk\"ahler manifold $X$ (or more generally a birational map on a minimal projective terminal variety)
is isomorphic in codimension one (cf. \cite{Huybrechts03}), it induces an isomorphism on $\NS(X)$.
So we can define its first dynamical degree as $d_1(g) := \rho(g^* | \NS_\bbC(X))$.
We say $g$ (resp. $G$) is of {\it null entropy} if $d_1(g) = 1$ (resp. $d_1(g) = 1$ for every $g \in G$).

\begin{thm}\label{HK equiv}
Let $X$ be a $2n$-dimensional projective hyperk\"ahler manifold of type $K3^{[n]}$ or of type generalized Kummer.
Let $G$ be an infinite subgroup of $\Bir(X)$.
Then $G$ is of null entropy if and only if there is a rational Lagrangian fibration (see {\rm\ref{rlf}}) $\phi:X\ratmap \bbP^n$ such that the birational action of $G$ on $X$ descends to a
biregular action on $\bbP^n$, i.e., such that the
following diagram commutes:
\[\xymatrix{
  X \ar@{-->}[r]^{g} \ar@{-->}[d]_{\phi} & X \ar@{-->}[d]^{\phi} \\
  \bbP^n \ar[r]^{\alpha(g)} & \bbP^n   }
\]
where $g$ is an element of $G$ and $\alpha$ is a group homomorphism from $G$ to $\Aut(\bbP^n)$.
\end{thm}

A birational map $g \in \Bir(X)$ which is isomorphic in codimension one, is said to be {\it polarized by a divisor} $D$
if $g^*D \equiv \lambda D$ (numerical equivalence) for some $\lambda > 0$.

Our next result is a hyperk\"ahler analogue of Theorem \ref{Surface auto} for surfaces.
See Remark \ref{rHK auto} for the converse of Theorem \ref{HK auto}.

\begin{thm}\label{HK auto}
Let $X$ be a projective hyperk\"ahler manifold. Assume $D$ is a numerically nonzero $\bbR$-divisor
such that $q_X(D) \ge 0$ for the Beauville-Bogomolov quadratic form $q_X$ (this holds when $D$ is movable or nef).
Assume further that
$g_i \in \Bir(X)$ ($i = 1, 2$)
are of positive entropy and polarized by $D$.
Then $g_1^{t_1}=g_2^{t_2}$ holds in $\Bir(X)$ for some $t_i \in \bbZ \setminus \{0\}$.
\end{thm}

The first assertion below is due to \cite[Theorem 2.1 (2) and Proposition 2.9]{Oguiso07}.
Recall that a group $G$ is called {\it virtually solvable} (resp. {\it virtually unipotent} or {\it virtually abelian}),
if a finite-index subgroup of $G$ is solvable (resp. unipotent or abelian).
Note that (virtually) unipotent groups are defined only for linear algebraic groups. 
A group $G$ is called {\it almost infinite cyclic},
if a finite-index subgroup of $G$ is an infinite cyclic group.

\begin{proposition}\label{PropA}
Let $X$ be a smooth projective surface and let $G \le \Aut(X)$ be a group such that the induced
action $G | \NS(X)$ is an infinite group, and $G$ is of null entropy.
Then $G | \NS(X)$ is virtually abelian of rank $s \le \rho(X) - 2$.
Further, $s \le 18$, unless $X$ has Kodaira dimension $\kappa(X) = 1$.
\end{proposition}

\begin{remark}
For a projective hyperk\"ahler manifold $X$ the same proof in \cite{Oguiso07} implies the first assertion in Proposition \ref{PropA}.
Further,
$$s \le \rho(X) - 2 \le \dim H^{1,1}(X, \bbC) - 2 \le b_2(X) - 4,$$ with $b_2(X)$ the second Betti number of $X$.
For a hyperk\"ahler manifold $X$ of dimension four, it is known that $b_2(X) \in \{3, \dots, 8, 23\}$ (cf.~\cite[Main Theorem]{Guan}).
\end{remark}

We refer to \cite{Cantat11}, \cite{Cantat14}, \cite{DF}, \cite{Fa} and \cite{Oguiso14}
for hard and deep results on birational actions on complex surfaces which are related
to Theorem \ref{Surface equiv}.

\par \vskip 1pc
\noindent
{\bf Acknowledgement.} This paper is finalized when
the last-named author visited KIAS, Seoul in December 2014.
He thanks the institute for the support and warm hospitality.
He is also supported by an ARF of NUS. The authors thank the referees for the
valuable suggestions and references \cite{Cantat14}, \cite{Fa}, \cite{Guan} and \cite{Rat}.

\section{Preliminaries}

Let $X$ be a smooth projective variety (resp. a compact K\"ahler manifold).
The {\it N\'eron-Severi group} $\NS(X)$ is defined as the group of line bundles modulo algebraic equivalence.
Denote its rank by $\rho(X)$, which is also called the {\it Picard number} of $X$.
The N\'eron-Severi space $\NS_{\bbR}(X) := \NS(X) \otimes_{\Z} \bbR$ is the vector space of numerical equivalence classes of $\bbR$-divisors.

Denote the group of all automorphisms (resp. biholomorphisms) of $X$ by $\Aut(X)$.
Denote the group of all birational maps (resp. bimeromorphisms) of $X$ by $\Bir(X)$.
By $\Aut_0(X)$ we mean the {\it identity connected component} of $\Aut(X)$.

\begin{setup}\label{cone}
{\bf K\"ahler cone, nef cone and movable cone}
\vskip 1pc

Let $X$ be a compact K\"ahler manifold. Set $H^{1,1}(X, \bbR):=H^{1,1}(X, \bbC)\cap H^2(X, \bbR)$.
The {\it K\"ahler cone} $\cK(X) \subseteq H^{1,1}(X, \bbR)$ is the open convex cone of all K\"ahler classes on $X$.
The closure of the K\"ahler cone $\cK(X)$ in $H^{1,1}(X, \bbR)$ is denoted by $\overline \cK(X)$.
The {\it nef cone} is defined as $\Nef(X) := \overline \cK(X) \cap \NS_{\bbR}(X)$.
An element $D$ in $\Nef(X)$ is called {\it nef}.

A divisor $D$ in $\NS(X)$ is {\it movable} if the linear system $|D|$ has no fixed component, i.e., the base locus of this linear system has codimension at least two.
The closure of all movable divisor classes in $\NS_{\R}(X)$ is called the {\it movable cone} of $X$
and denoted as $\Mov(X)$.
It is known that $$\Nef(X) \subseteq \Mov(X) \subseteq \Eff(X),$$ where the {\it pseudo-effective cone} $\Eff(X)$ of $X$ is the closure of all effective divisor classes in $\NS_{\bbR}(X)$.
\end{setup}

\begin{setup}\label{rlf}
{\bf Rational Lagrangian fibrations on hyperk\"ahler manifolds}
\vskip 1pc

A {\it hyperk\"ahler manifold} is a simply-connected compact K\"ahler manifold $X$ such that $H^0(X, \Omega_X^2)$ is generated by an everywhere non-degenerate holomorphic $2$-form $\omega$.
Note that $X$ is automatically of even complex dimensional (say $\dim X = 2n$).
A surjective (holo)morphism $\phi: X \to S$ to a normal variety $S$ is said to be a {\it Lagrangian fibration}
if a general fiber is connected and {\it Lagrangian} (i.e., the dimension of this fibre is $n$ and the restriction of the holomorphic $2$-form $\omega$ to this fibre is trivial).
A dominant meromorphic map $\phi: X \ratmap S$ to a normal variety $S$ is said to be a {\it rational Lagrangian fibration} if there exists a bimeromorphic map $\tau: X \ratmap X'$
to another hyperk\"ahler manifold $X'$
(which is necessarily isomorphic in codimension one)
such that the composite map $\phi\circ \tau^{-1}: X' \to S$ is a Lagrangian fibration.
Thus $\phi^*$ is well-defined on the N\'eron-Severi group $\NS(S)$ of $S$.
\end{setup}

\begin{setup}
{\bf Positive cone and birational K\"ahler cone}
\vskip 1pc

Let $X$ be a compact hyperk\"ahler manifold.
The {\it positive cone} $\cC(X)$ in $H^{1,1}(X, \bbR)$ is the connected component
of the open cone $$\{\alpha\in H^{1,1}(X, \bbR) :  q_X(\alpha)>0\}$$
that contains the K\"ahler cone $\cK(X)$, where $q_X$ is the Beauville-Bogomolov quadratic form.
The closure of $\cC(X)$ in $H^{1,1}(X, \bbR)$ is denoted by $\overline \cC(X)$.

If $f:X'\ratmap X$ is a bimeromorphic map between two hyperk\"ahler manifolds,
then it is isomorphic in codimension one.
Hence the pullback of $f$ is well defined on $H^2(X, \bbC)$ and compatible with its Hodge structure,
the Beauville-Bogomolov quadratic form $q_X$ and the birational K\"ahler cone $\BK(X)$
(cf. \cite[Propositions 21.6 and 25.14]{Huybrechts03}).
Recall that the {\it birational K\"ahler cone} $\BK(X)$ of $X$ is defined as
$$\BK(X):=\bigcup_{\tau:X\ratmap X'} \tau^*\cK(X'),$$
where $\tau:X\ratmap X'$ runs through all bimeromorphic maps from $X$ to another hyperk\"ahler manifold $X'$.
Note that $\BK(X)$ is in general not a cone in $\cC(X)$, but its closure $\overline \BK(X)$ in $H^{1,1}(X, \bbR)$ is a closed convex cone contained in $\overline \cC(X)$.

There is a geometric characterization of the birational K\"ahler cone,
which states that $\alpha \in \overline\BK(X)$ if and only if $\alpha \in \overline \cC(X)$
and $q_X(\alpha,[D]) \ge 0$ for all uniruled prime divisors $D \subset X$ (cf. \cite[Proposition 28.7]{Huybrechts03}).
Furthermore, we have $$\Mov(X) = \overline\BK(X) \cap \NS_{\bbR}(X).$$
\end{setup}

The following two lemmas on group theory are useful in the future.

\begin{lemma} \label{Group}
~\begin{enumerate}
\item[(1)]
Let $G$ be a group, $H \lhd G$ a finite normal subgroup and $g_1,\, g_2 \in G$. Suppose that $\bar{g}_1 = \bar{g}_2$ in $G/H$.
Then there exists a positive integer $s$ such that $g_1^s = g_2^s$.
\item[(2)]
A group $G$ is almost infinite cyclic if and only if there is a finite-index subgroup $G_1$ of $G$
such that $G_1/G_2$ is infinite cyclic for some finite $G_2 \lhd G_1$.
\end{enumerate}
\end{lemma}

\begin{proof}
For part (1), by the assumption, $g_1^n g_2^{-n} \in H$ for all $n \in \bbZ$. Since $H$ is finite, there exist $m < n$ such that $g_1^m g_2^{-m} = g_1^n g_2^{-n}$.
Then $g_1^s = g_2^s$ for $s := n-m$.

Part (2) is easy; see \cite[Lemma 2.4]{CWZ14}.
\end{proof}

Note that the lemma below or \cite[Lemma 2.7]{CWZ14} has been generalized in \cite[Lemma 5.5]{Di} without assuming $G$ to be a subgroup of an algebraic group.
But the following simple form with simple proof as in \cite[Lemma 2.7]{CWZ14} is enough for us.

\begin{lemma}\label{ext}
Let $G$ be a subgroup of an algebraic group $\hat{G}$ (which is an extension of an abelian variety by a linear algebraic group).
Consider the exact sequence
$$ 1 \to N \to G \to Q \to 1.$$
Suppose that both $N$ and $Q$ are virtually solvable. Then so is $G$.
\end{lemma}

\begin{proof}
This is implicitly proved in \cite[Lemma 2.7]{CWZ14}.
We just use the fact that the Zariski-closure of $N$ in $\hat{G}$ has only finitely many connected components.
Then the same argument in \cite{CWZ14} gives the proof.
\end{proof}

At the end of this section we quote the following lemmas which will be used later.

\begin{lemma}[{cf. \cite[Theorem 4.8]{Fujiki78}, \cite[Proposition 2.2]{Lieberman78}, or \cite[Lemma 2.6]{CWZ14}}] \label{Finite res}
Let $X$ be a smooth projective variety (resp. a compact K\"ahler manifold),
and $L := \NS(X)/(\torsion)$ (resp. $H^2(X, \bbZ)/(\torsion)$).
Then a group $G \le \Aut(X)$ has finite induced action $G | L$ if and only if the index $|G : G \cap \Aut_0(X)|$ is finite.
\end{lemma}

\begin{lemma}[{cf. \cite[Corollary 2.7]{Oguiso08-JDG}, or \cite[\S 9]{Huybrechts99}}
for the case of the action of $G$ on $H^2(X, \bbZ)$] \label{finite kernel}
Let $X$ be a projective hyperk\"ahler manifold and let $G \le \Bir(X)$.
Then the homomorphism
$r_{\NS}: G \rightarrow \OO(L)\textrm{ (cf. {\rm \ref{hyper_l}})},\ g \mapsto g^* | L$ has finite kernel, where $L:=\NS(X)$.
\end{lemma}

\begin{lemma}[{cf. \cite[Application 1.4]{FZ13}}] \label{irrat}
Let $X$ be a compact K\"ahler surface with first
Betti number $b_1(X) = 0$. Suppose that $\Aut_0(X) \ne 1$. Then $|\Aut(X) : \Aut_0(X)| < \infty$ and $X$ is projective. In particular, $\Aut(X)$ is of null entropy.
\end{lemma}

\section{Automorphisms of hyperbolic lattices and applications}

\begin{setup}\label{hyper_l}
{\bf Hyperbolic lattices and their symmetries}
\vskip 1pc

By a lattice $L$, we mean a free abelian group $L \cong \bbZ^{\oplus r}$
admitting a non-degenerate symmetric bilinear form $L\times L\to\bbZ$, denoted as $(\,\cdot\,,\,\cdot\,)$.
The signature of $L$ is defined as the signature of $L_\bbR := L\otimes_\bbZ \bbR$.
The lattice $L$ is called {\it hyperbolic}
if the signature of $L$ is $(1, r-1)$ .

The {\it positive cone} $\cC(L_{\bbR})$ is one of the two connected components of $$\{x\in L_\bbR : (x,x)>0\}.$$
Denote the boundary (resp. closure) of $\cC(L_{\bbR})$ by $\partial\cC(L_{\bbR})$ (resp. $\overline{\cC}(L_{\bbR})$).
Note that there is no ambiguity when applied to surfaces or hyperk\"ahler manifolds.
In fact, the positive cone is always chosen by us so that it contains an ample class or a K\"ahler class.

We denote the {\it group of isometries} of $L$ by
$$\OO(L):=\{g\in \Aut(L) : (gx,gy)=(x,y), \, \forall x,y\in L\}.$$
The subgroup $\OO(L)'$ of $\OO(L)$ which preserves the positive cone has index two.
\end{setup}

\begin{remark} \label{ns_lattice}
When $X$ is a smooth projective surface (resp. a projective hyperk\"ahler manifold), the intersection form (resp. the Beauville-Bogomolov quadratic form $q_X$) on $X$
gives a non-degenerate symmetric bilinear form on the lattice $L:=\NS(X)/(\torsion)$
(resp. $L:=\NS(X)$; note that the N\'eron-Severi group is torsion free since $X$ is simply connected).
In either case, $L$ is a hyperbolic lattice of signature $(1, \rho(X) - 1)$ (cf. \cite[Proposition 26.13]{Huybrechts03}).
The above $L$ is called the {\it N\'eron-Severi lattice} of $X$.
\end{remark}

\begin{setup}
{\bf Spectral radius and entropy}
\vskip 1pc

Let $k = \bbZ$ or a subfield of $\bbC$, $V$ a finite $k$-module,
and $\varphi : V \to V$ a $k$-linear endomorphism.
Set $V_{\bbC} := V \otimes_k \bbC$.
Use the same $\varphi$ to denote its extension to a $\bbC$-linear
endomorphism $\varphi : V_{\bbC} \to V_{\bbC}$.
Define the {\it spectral radius} of $\varphi$ as follows
$$\rho(\varphi) := \max \{|\lambda| : \lambda \in \bbC \text{ is an eigenvalue of }
\varphi : V_{\bbC} \to V_{\bbC} \} .$$
Define the {\it entropy} of $\varphi$ as
$h(\varphi) := \log \rho(\varphi)$.

Let $G \le \OO(L)$ for some lattice $L$. We define the {\it null-subset} of $G$ as
$$N(G) :=\{g\in G : g \text{ is of null entropy, i.e., } h(g)=0, \text{ or equivalently } \rho(g) = 1 \}.$$
\end{setup}

The following generalized Perron-Frobenius theorem is due to Birkhoff.

\begin{lemma}[{cf. \cite{Birkhoff67}}] \label{Bir}
Let $C$ be a strictly (i.e., salient) convex closed cone of a finite-dimensional
$\bbR$-vector space $V$ such that $C$ spans $V$ as a vector space.
Let $g : V \to V$ be a $\bbR$-linear endomorphism such that $g(C) \subseteq C$.
Then the spectral radius $\rho(g)$ is an eigenvalue of $g$ and
there is an eigenvector $v_g \in C$ corresponding to the eigenvalue $\rho(g)$.
\end{lemma}

We now prepare some general results on $\OO(L)$ for a hyperbolic lattice $L$.

\begin{lemma}[{cf. \cite[Chapter 5]{Rat}, or \cite[Proposition 2.2]{Oguiso07}, or
\cite[Theorem 2.2]{CWZ14}}] \label{Oguiso1}
Let $L$ be a hyperbolic lattice. Suppose that $G \le \OO(L)'$ is of null entropy.
Then $G$ contains $U$ as a finite-index normal subgroup,
where
$$U = U(G) := \{g \in G : g \ \text{\rm is unipotent as an element in} \, \GL(L_{\bbC})\} .$$
\end{lemma}

\begin{proof}
This is implicitly proved in \cite[Proposition 2.2]{Oguiso07}. Indeed, by \cite[Lemma 2.5]{Oguiso07}, $U$ is a subgroup of (and hence normal in) $G$.
Replacing $G$ by a finite-index subgroup, we may assume that the Zariski-closure of $G$ in $\GL(L_{\bbC})$ is connected.
Now the quotient group $G/U$ can be embedded into an algebraic torus (and hence in some general linear group)
and every element of it is of finite order (bounded by a constant depending only on $\rank(L)$) by Kronecker's theorem.
Note also that every subgroup of a general linear group over $\bbC$ with bounded exponent is a finite group by the classical Burnside's Theorem.
Hence $|G : U| < \infty$ as required.
\end{proof}

The result below follows from the classification of elliptic, parabolic or loxodromic elements;
see for instance \cite{Cantat14}, or the proof of \cite[Lemma 2.12]{Zhang08}.

\begin{lemma} [{cf. \cite{Cantat14}, or proof of \cite[Lemma 2.12]{Zhang08}}] \label{e.v.}
Let $L$ be a hyperbolic lattice with $\cC(L_{\bbR})$ the positive cone. Let $g \in \OO(L)'$ such that
$g(v) = \lambda v$ for some $0 \ne v \in \overline \cC(L_{\bbR})$.
\begin{enumerate}
\item[(1)]
If $\lambda = 1$, then $g$ is of null entropy.
\item[(2)]
If $\lambda > 1$ (resp. $<1$), then
$g$ is of positive entropy, $v$ is parallel to the $v_g$ (resp. $v_{g^{-1}}$) in Lemma {\rm \ref{Bir}} and $\rho(g) = \lambda$ (resp. $\rho(g^{-1}) = \lambda^{-1}$).
\item[(3)]
Suppose that $\lambda \ne 1$. Then
$\rho(g)$ is a Salem number or a quadratic integer
(and hence an irrational algebraic integer),
the $v_g$ in Lemma {\rm \ref{Bir}} in the current case is unique up to a scalar, and $\rho(g) = \rho(g^{-1})$.
\end{enumerate}
\end{lemma}

\begin{proof}
For the assertions (1) and (2), the proof is similar to \cite[Lemma 2.12]{Zhang08}.
We only have to consider the case where $g$ is of positive entropy.
Hence $g^{-1}$ is also of positive entropy, because $\det(g) = \pm 1$ and $g$ is defined over the integral lattice $L$.
By Lemma \ref{Bir}, there are nonzero eigenvectors $v_{g^{\pm 1}} \in \overline\cC(L_\bbR)$ such that
$$g^{\pm 1}(v_{g^{\pm 1}}) = \rho(g^{\pm 1}) v_{g^{\pm 1}} .$$
It suffices to show the claim that $v$ is parallel to one of $v_{g^{\pm 1}}$.
Suppose the contrary that this claim is not true.
The hyperbolicity of $L$ and the Hodge index type theorem imply that
$$0 < (v, v_g) = (g(v), g(v_g)) = \lambda \rho(g) (v, v_g) .$$
Hence $\lambda = 1/\rho(g) < 1$.
By the same reasoning,
$$0 < (v, v_{g^{-1}}) = (g^{-1}(v), g^{-1}(v_{g^{-1}})) = \lambda^{-1} \rho(g^{-1}) (v, v_{g^{-1}}) .$$
Hence $\lambda = \rho(g^{-1}) > 1$, contradicting the previous outcome.
Thus the claim is true. Hence the assertions (1) and (2) are true.

For the assertion (3), $\rho(g)$ is a Salem number by \cite[Proposition 2.5]{Oguiso06}.
Hence $\rho(g) = \rho(g^{-1})$. The uniqueness follows from the above argument for assertion (2).
\end{proof}

\begin{lemma}[{cf. \cite[\S 5.5]{Rat} or \cite[Theorem 2.1]{Oguiso07}}] \label{Oguiso}
Let $L$ be a hyperbolic lattice and let $G \le \OO(L)'$. Suppose that $G$ is of null entropy
and $G$ is an infinite group. Then we have:

\begin{enumerate}
\item[(1)]
There is a unique (up to scalars) nonzero $v \in \overline{\cC}(L_{\bbR})$
such that $g(v) = v$ for all $g \in G$.
Moreover, $v^2 = 0$ and $v$ can be taken to be in the integral lattice $L$.

\item[(2)]
Suppose that $W$ is a non-trivial $G$-invariant closed subcone of $\overline{\cC}(L_{\bbR})$. Then $v$ belongs to $W$.

\item[(3)]
If $L = \NS(X)/(\torsion)$ (resp. $\NS(X)$) for some smooth projective surface (resp. projective hyperk\"ahler manifold) $X$
and the above $G$ equals $H | L$ for some group $H \le \Aut(X)$, then $v$ is a nef divisor class.

\item[(4)]
If $L = \NS(X)$ for some projective hyperk\"ahler manifold $X$ and the above $G$ equals $H | L$ for some
group $H \le \Bir(X)$, then $v$ belongs to the movable cone $\Mov(X)$.
\end{enumerate}
\end{lemma}

\begin{proof}
Part (1) is contained in \cite[\S 5.5]{Rat} or \cite{Oguiso07}.
Indeed, \cite[Lemma 2.8]{Oguiso07} implies the claim that there is no $v \in L_{\bbQ}$
such that $v^2 > 0$ and $g(v) = v$ for all $g \in G$.
There is also no such $v$ in $L_{\bbR}$. Indeed, as noticed in \cite{Oguiso07},
the eigenspace $V(g, 1)$ of eigenvalue $1$
is defined over $\bbQ$, hence the intersection $\cap_{g \in G} \, V(g, 1)$ is also defined over $\bbQ$.
Thus the existence of such $v$ in $L_{\bbR}$
(and the density of $\bbQ$ in $\bbR$) would imply the same for a $v$ in $L_{\bbQ}$, contradicting the claim.
By the claim above and \cite[Lemma 2.8]{Oguiso07}
there indeed exists a unique nonzero ray $\bbR_{>0} \, v \subseteq \overline{\cC}(L_{\bbR})$
such that $g(v) = v$ for all $g \in G$;
further $v^2 = 0$ and $v$ can be taken to be in $L$.

Parts (3) and (4) are consequences of part (2) applied to the nef cone $\Nef(X)$ and the movable cone $\Mov(X)$, respectively.
We also note that a bimeromorphic map of a hyperk\"ahler manifold $X$ is an isomorphism in codimension one,
hence induces an isomorphism of $\NS(X)$ and preserves the movable cone $\Mov(X)$.

For part (2), note that $G$ is virtually unipotent and hence virtually solvable
(cf.~\cite[Proof of Theorem 2.1]{Oguiso07} or \cite[Proof of Theorem 2.2]{CWZ14}).
Thus $G_1(v') \subseteq \R_{> 0} v'$ for some nonzero element $v'\in W$
and a finite-index subgroup $G_1$ of $G$,
by the cone theorem of Lie-Kolchin type (cf. \cite[Theorem 1.1]{KOZ09}).
By Lemma \ref{e.v.} and the assumption $G=N(G)$, we have $G_1 (v') = v'$.
Now apply the uniqueness property in part (1) to $G_1$, we conclude that $v\in \bbR_{>0}v' \subseteq W$.
This proves part (2) and the whole lemma.
\end{proof}

\begin{lemma}\label{fin}
Let $L$ be a hyperbolic lattice and let $G \le \OO(L)'$.
Suppose that the null-subset $N(G)$ is a subgroup of (and hence normal in) $G$ and $G \ne N(G)$.
Then $N(G)$ is finite.
\end{lemma}

\begin{proof}
Suppose the contrary that $|N(G)| = \infty$. By Lemma \ref{Oguiso} or \cite[Lemma 2.8]{Oguiso07},
there is a nonzero $v \in \partial \cC(L_{\bbR})$ such that $h(v) = v$ for any $h\in N(G)$.
Further, the ray $\bbR_{>0}v \subseteq \partial\cC(L_{\bbR})$ is unique and defined over $\bbZ$.
Now for any $g\in G$ and $h\in N(G)$, we must have $h(g(v)) = g(v)$ since $N(G) \lhd G$.
Thus $g(v) = r_g v$ for some $r_g\in \bbR_{>0}$ by the uniqueness of the above ray.
Moreover, $r_g$ must be rational since $v$ is in the integral lattice $L$.
Take one $g \in G$ which is of positive entropy.
Since $g(v) = r_g v$, we have $r_g = \rho(g)^{\pm 1}$ by Lemma \ref{e.v.}.
But $\rho(g)$ is a Salem number or a quadratic integer and hence is irrational, contradicting the rationality of $r_g$.
Hence $N(G)$ is finite.
\end{proof}

The result below is contained in \cite[Theorems 5.5.9 - 5.5.10]{Rat}. See also \cite[Theorem 3.2]{Cantat14}
or \cite[Theorem 3.1]{Zhang08} for $\NS(X)$ of a surface $X$.

\begin{lemma}[{cf. \cite[Theorems 5.5.9 - 5.5.10]{Rat}}] \label{Lattice equiv}
Let $L$ be a hyperbolic lattice and let $G \le \OO(L)'$. Then the following are equivalent:
\begin{enumerate}
\item [(1)] $G$ is virtually solvable.
\item [(2)] Replacing $G$ by a finite-index subgroup,
there is a real vector $v \in \overline \cC(L_{\bbR})\setminus\{0\}$ such that $G(v) \subseteq \bbR_{>0} v$.
\item [(3)]
Replacing $G$ by a finite-index subgroup, we have $N(G) \lhd G$ and $G/N(G) \cong \bbZ^{\oplus r}$
for some nonnegative integer $r \le 1$.
\end{enumerate}
\end{lemma}
\begin{proof}
For readers' convenience, we include the proof.

$(1)\Rightarrow(2)$ Replacing $G$ by a finite-index subgroup, we may assume that the Zariski closure of $G | L_{\bbC}$ in $\GL(L_{\bbC})$ is connected.
Then the assertion (2) is just a consequence of Theorem of
Lie-Kolchin type for a cone (cf. \cite[Corollary 2.3]{KOZ09}).
\medskip

$(2)\Rightarrow(3)$ Replacing $G$, we may assume that there exists a $\lambda_g\in \bbR_{>0}$ such that $g(v) = \lambda_g v$ for any $g\in G$ .
Define a group homomorphism
$$\phi: G \rightarrow (\bbR, +),\ g \mapsto \log\lambda_g.$$

Now the assertion (3) follows from the two claims below.

\begin{claim} \label{norm Null}
$\Ker\phi=N(G)=\{g\in G : \rho(g)=1\}$. In particular, $N(G) \lhd G$.
\end{claim}
\begin{proof} \renewcommand{\qedsymbol}{}
Clearly, $N(G) \subseteq \ker\phi$.
Conversely, suppose $g \in \ker \phi$.
Then $g(v) = v$. Hence $g$ is of null entropy, i.e., $g \in N(G)$, by Lemma \ref{e.v.}.
\end{proof}

\begin{claim}\label{discrete}
$\phi(G)$ is discrete in the additive group $\bbR$, hence $\phi(G) = 0$ or $\bbZ$.
\end{claim}
\begin{proof} \renewcommand{\qedsymbol}{}
It suffices to show that 0 is an isolated point in $\phi(G)$.
For any fixed positive real number $\delta$, consider the set
$$G_\delta:=\{g\in G : |\log\lambda_g| <\delta\} .$$
For any $g\in G_\delta$, $\lambda_g^{\pm 1}$ is bounded by $e^\delta$.
If $g \in N(G)$, then $\lambda_g = 1$.
Suppose that $g$ is of positive entropy.
Then $\lambda_g \ne 1$ and $\rho(g) = \lambda_g^{\pm 1}$ (cf. Lemma \ref{e.v.}).
The minimal polynomial of $g$ is a Salem polynomial
with two real roots $\rho(g)^{\pm 1}$ and other roots on the unit circle (cf. \cite[Proposition 2.5]{Oguiso06},
\cite[Lemma 2.7]{Zhang08}).
Hence the coefficients of these minimal polynomials are all bounded.
Note that these minimal polynomials are defined over $\bbZ$ with degree bounded by $\rank(L)$.
So there are only finitely many such minimal polynomials for all $g\in G_\delta$.
Therefore, the set of all eigenvalues (especially $\lambda_g$)
of such $g \in G_{\delta}$ is finite. Thus $0$ is an isolated point in $\phi(G)$.
\end{proof}

$(3)\Rightarrow(1)$
Replacing $G$, we may assume that
$G/N(G)$ is cyclic and hence solvable.
$N(G)$ is virtually unipotent by Lemma \ref{Oguiso1} and hence virtually solvable.
Thus $G$, regarded as a subgroup of $\GL(L_{\bbC})$
is virtually solvable, by Lemma \ref{ext}.
So the assertion (1) is true.
\end{proof}

\begin{corollary}\label{almost infinite}
Assume one of the equivalent conditions in Lemma {\rm \ref{Lattice equiv}} holds and further that $G \ne N(G)$.
Then replacing $G$ by a finite-index subgroup, we have $N(G)$ is finite and $G$ is almost infinite cyclic.
\end{corollary}

\begin{proof}
By the assumption and Lemma \ref{Lattice equiv}, we may assume that $G/N(G) \cong \bbZ$ after replacing $G$.
Since $G \ne N(G)$, our $N(G)$ is finite by Lemma \ref{fin}.
Hence $G$ is almost infinite cyclic; see Lemma \ref{Group}.
\end{proof}

The following result is a direct consequence of \cite{Rat}. And it is the lattice-theoretical counterpart of
the result on surfaces or hyperk\"ahler manifolds we will state later on.

\begin{proposition} [{cf. \cite[Theorem 5.5.8]{Rat}}] \label{Lattice}
Let $L$ be a hyperbolic lattice and let $g_1, g_2 \in \OO(L)'$. Suppose that both $g_i$ are of positive entropy, and
there is a nonzero element $v \in \overline{\cC}(L_\bbR)$, such that $g_i(v)=\lambda_i v$ for some real numbers $\lambda_i > 0$.
Then $g_1^{t_1} = g_2^{t_2}$ for some $t_i \in \bbZ \setminus \{0\}$.
\end{proposition}

\begin{proof}
This follows from \cite[Theorem 5.5.8]{Rat}.
We prove it for the convenience of readers.
Let $G=\langle g_1, g_2 \rangle$.
Then $G$ has a common real eigenvector $v \in \overline \cC(L_\bbR)\setminus\{0\}$ and $G \ne N(G)$.
Applying Lemma \ref{Lattice equiv} and replacing $G$ by a finite-index subgroup, we have
$$\pi:G\rightarrow G/N(G) \cong \bbZ.$$
Further, $N(G)$ is finite by Lemma \ref{fin}.
Thus $\pi(g_1^{n_1})=\pi(g_2^{n_2})$ for some $n_1, n_2\in \bbZ\setminus \{0\}$.
Now the proposition follows from Lemma \ref{Group}.
\end{proof}

\begin{setup}
{\bf Consequences on automorphisms of surfaces and hyperk\"ahler manifolds}
\vskip 1pc

Now we apply the above results on hyperbolic lattice to the N\'eron-Severi lattice (see Remark \ref{ns_lattice}) of a projective surface or a hyperk\"ahler manifold.
\end{setup}

\begin{proposition} \label{Finite Null}
Let $X$ be a smooth projective surface and $G \le \Aut(X)$. Suppose the null-subset $N(G)$ is a proper subgroup of (and hence normal in) $G$. Then we have:
\begin{enumerate}
\item[(1)]
  $N(G) | \NS(X)$ is a finite group.
\item[(2)]
  Suppose that $|N(G)|=\infty$. Then $X$ is an abelian surface,
  $H := N(G) \cap \Aut_0(X) = G \cap \Aut_0(X)$ is Zariski-dense in $\Aut_0(X) \, (\cong X)$ and the index $|N(G) : H| < \infty$.
\end{enumerate}
\end{proposition}

\begin{proof}
Set $L := \NS(X)/(\torsion)$.
Replacing $G$ by a suitable finite-index subgroup we may assume that the Zariski-closure of $G | L_{\bbC}$ in $\GL(L_{\bbC})$ is connected. Note that
$N(G)$ is the inverse of $N(G | L)$ via the natural homomorphism $G \to G | L \le \OO(L)'$,
and $N(G) | L = N(G | L)$.
Part (1) follows directly from Lemma \ref{fin}.

For part (2), suppose that $N(G)$ is infinite (but $N(G) | L$ is finite by part (1)).
Then the Kodaira dimension $\kappa(X) \le 1$ since the automorphism group of a variety of general type is finite.
Also $H:= N(G) \cap \Aut_0(X)$ is equal to $G \cap \Aut_0(X)$ and has finite-index in $N(G)$, by Lemma \ref{Finite res}.
Thus $\Aut_0(X) \ne 1$, so $X$ is neither birational to a $K3$ surface nor to an Enriques surface. Since $G \ne N(G)$, our $X$ is not a rational surface by Lemma \ref{irrat}.

The assumption $G \ne N(G)$ and Lemma \ref{e.v.} or \cite[Lemma 2.12]{Zhang08}
imply that no fibration on $X$ is equivariant under the action of
$G$ or its finite-index subgroup,
and hence $X$ is birational to an abelian surface.
(You may also appeal to Cantat's classification of surfaces with at least one automorphism of positive entropy; see \cite{Cantat99}.)

Let $G_0$ be the Zariski closure of $H$ in $\Aut_0(X)$. Then $G_0$ is normalized by $G$.
Thus $X$ is dominated by $G_0$, by the proof of \cite[Lemma 2.14]{Zhang09-Invent} and since $X$
and hence its $G$-equivariant blowup
have no non-trivial $G$-equivariant fibration.
Note that the albanese map
$\alb_X : X \to A:=\Alb(X)$
is surjective, birational and necessarily $\Aut(X)$-equivalent.
$G_0$ induces an action on $A$ and
we denote it by $G_0|A$. Since $G_0|A$ also has a Zariski-dense open orbit in $A$,
the closedness of $G_0$ implies that
$G_0|A = \Aut_0(X) \, (\cong A)$.
Let $B\subset A$ be the locus over which $\alb_X$ is not an isomorphism. Note that $B$ and $\alb_X^{-1}(B)$ are $G_0$-stable.
Since $G_0|A = \Aut_0(X)$ consists of all translations on $A$, we have $B = \emptyset$. Thus $X$ is an abelian surface.
This proves Proposition \ref{Finite Null}.
\end{proof}

Below are consequences of Proposition \ref{Finite Null} and hyperk\"ahler analogue.
Theorem \ref{Surface cyclic} is part of \cite[Theorem 3.2 and Remark 3.3]{Cantat14}.
It also follows from Lemmas \ref{fin} and \ref{Lattice equiv} and Lemma \ref{Group}.
The proof of Theorem \ref{HK cyclic} is similar to the part (1) of Theorem \ref{Surface cyclic},
but Lemma \ref{finite kernel} is also used.

\begin{thm} [{cf. \cite[Theorem 3.2, Remark 3.3]{Cantat14}}] \label{Surface cyclic}
Let $X$ be a smooth projective surface.
Suppose that $G \le \Aut(X)$ is not of null entropy and that $G|\NS_{\bbC}(X)$ is virtually solvable.
Then after replacing $G$ by a finite-index subgroup, we have:
\begin{enumerate}
\item[(1)]
$N(G) | \NS(X)$ is finite and $G | \NS(X)$ is almost infinite cyclic.
\item[(2)]
$G$ is almost infinite cyclic, unless $X$ is an abelian surface and $G \cap \Aut_0(X)$ is Zariski-dense in $\Aut_0(X)$.
\end{enumerate}
\end{thm}
\begin{proof}
For readers' convenience, we include the proof.
Let $L = \NS(X) /(\torsion)$.
Note that $N(G)$ is the inverse of $N(G|L)$ via the natural homomorphism $G \to G | L$,
and $N(G) | L = N(G | L)$.
Replacing $G$ by a finite-index subgroup, we may assume that $G | L$ is solvable and
the Zariski closure of $G | L_{\bbC}$ in $\GL(L_{\bbC})$ is connected. Then,
as in Lemma \ref{Lattice equiv},
$N(G | L) \lhd G | L$ and
$$G/N(G) \cong (G|L)/(N(G|L)) \cong \bbZ .$$
Since $N(G) \ne G$, Lemma \ref{fin} implies that $N(G) | L$ and hence $N(G) | \NS(X)$ are finite.
Hence part (1) is true; see Lemma \ref{Group}.

If $N(G)$ is finite, then $G$ is almost infinite cyclic by Lemma \ref{Group}.
Otherwise,
we can apply Proposition \ref{Finite Null} and conclude part (2).
\end{proof}

\begin{thm}\label{HK cyclic}
Let $X$ be a projective hyperk\"ahler manifold.
Suppose that $G \le \Bir(X)$ is not of null entropy and that $G | \NS_\bbC(X)$ is virtually solvable.
Then after replacing $G$ by a finite-index subgroup, $N(G)$ is finite and $G$ is almost infinite cyclic.
\end{thm}
\begin{proof}
We argue as in Theorem \ref{Surface cyclic} (1) and also
apply Lemma \ref{finite kernel}.
\end{proof}

\section{Results for surfaces and proofs of results in the introduction}

For a surface $X$, a group $G\le \Aut(X)$ being of null entropy has a clear geometric interpretation as follows.
The crucial case has been dealt with in \cite[Theorem 2]{Gizatullin81} and \cite[Theorem 2.1]{Oguiso07}.
\cite[Theorem 2.11]{Cantat14} deals with the cyclic group case, but the general case is similar due to
the canonicity of the fibrations involved.
We state it for comparison with Theorem \ref{HK equiv}.

\begin{thm}[{cf. \cite[Theorem 2.11]{Cantat14}}] \label{Surface equiv}
Let $X$ be a smooth projective surface and let $G \le \Aut(X)$ be a group such that
the induced action $G|\NS(X)$ is an infinite group.
Then $G$ is of null entropy if and only if there is a $G$-equivariant fibration $X \to B$ onto a nonsingular projective curve $B$.
\end{thm}

The remark below shows that one can also handle the case where
$G$ is infinite while $G | \NS(X)$ is finite. See neat and exhaustive results in
\cite{Oguiso14} when $G \le \Bir(X)$ is cyclic.

\begin{remark}
Suppose that $X$ is a compact K\"ahler manifold (resp. a smooth projective
variety) and $G \le \Aut(X)$ is infinite while $G | H^{1,1}(X, \bbC)$ (resp. $G | \NS(X)$) is finite.
By Lemma \ref{Finite res} and its proof in \cite[Lemma 2.6]{CWZ14}, $G_0 := G \cap \Aut_0(X)$ is infinite and has
finite index in $G$.
Let $\overline{G}_0$
be the Zariski-closure of $G_0$ in $\Aut_0(X)$, and $L(\overline{G}_0)$ the linear part of $\overline{G}_0$;
see \cite{Fujiki78} or \cite{Lieberman78}.
Then either the graph of the quotient map $X \dashrightarrow X/\overline{G}_0$
or $X \dashrightarrow X/L(\overline{G}_0)$
as in \cite[Lemma 4.2]{Fujiki78}
gives a non-trivial $G$-equivariant fibration, or $X$ is almost
homogeneous under the action of $L(\overline{G}_0)$, or the albanese map
$\alb_X:X \to A := \Alb(X)$ is $\Aut(X)$- (and hence $\overline{G}_0$-) equivariant bimeromophic
with $\overline{G}_0$ acting densely on both the domain and codomain.
In the second case, $X$ is a unirational variety with $-K_X$ big and $|\Aut(X) : \Aut_0(X)| < \infty$
(cf. \cite[Theorem 1.2]{FZ13}). In the last case, the exceptional locus of
$X \to A$ (resp. its image in $\Alb(X)$) is stable under the action of $\overline{G}_0 \le \Aut_0(X)$
(resp. $\overline{G}_0 = \Aut_0(A) \cong A$), hence this locus is empty, i.e.,
$X \to A$ is an isomorphism.
\end{remark}

\begin{proof}[Proof of Theorem~{\rm\ref{Surface equiv}} (for the convenience of readers)]
Set $L := \NS(X)/(\torsion)$.

For the `if' part: Take a fibre $F$ of the fibration $X \to B$.
The $G$-equivariance of the fibration implies that $g^*[F] = [F]$ for every $g \in G$.
Hence $G$ is of null entropy by \cite[Lemma 2.12]{Zhang08} or Lemma \ref{e.v.}.

For the `only if' part: Suppose that $|G|L| = \infty$ and $G$ is of null entropy. Thus $X$ is not of general type and hence the Kodaira dimension $\kappa(X) \le 1$.
By Lemma \ref{Oguiso} or \cite[Theorem 2.1]{Oguiso07},
there is a unique (up to scalars) nonzero nef divisor $D$ with $D^2 = 0$ and $g^*[D] = [D]$ for all $g \in G$.
Further $D$ can be chosen to be integral.

If $\kappa(X) = 1$ (i.e., properly elliptic surfaces), or $X$ is a hyperelliptic surface and hence has a unique elliptic fibration onto $\Alb(X)$, or $X$ is an irrational ruled surface, then $X$ has a typical fibration which is clearly $G$-equivariant.

It remains to consider (blowups of) $K3$, Enriques, abelian and rational surfaces. We may assume that $X$ is minimal unless it is rational.

If $X$ is a $K3$ surface then the Riemann-Roch theorem
implies that the above nef $D$ is parallel to a fibre of an elliptic fibration and Theorem \ref{Surface equiv} is true.
The Enriques case can be reduced to its $K3$ cover and the pullback of $D$.

Suppose that $X$ is rational. We may assume that the pair $(X, G)$ is minimal. Then,
by \cite[Theorem 2, and p.~104]{Gizatullin81}, $K_X^2 = 0$,
the Iitaka $D$-dimension $\kappa(X, -K_X) = 1$ and for some $m \ge 1$, $|{-}mK_X|$ defines an elliptic fibration which is clearly $G$-equivariant.

Finally, suppose that $X$ is an abelian surface.
By Lemma \ref{Oguiso1}, $G | L$ is virtually unipotent.
Since $G | L$ is infinite,
some $g_0 \in G$ restricts to a unipotent element $g_0^* | L$ which is non-trivial and hence of infinite order.
Since $G$ is of null entropy, we may also assume that $g_0^* | H^0(X, \Omega_X^1)$ is non-trivial and unipotent,
after replacing $g_0$ by its power and using Kronecker's theorem.
Write $g_0 = T_a \circ h$ with $T_a$ a translation and $h$ a group automorphism.
Then $\ker(h - \id_X)$ contains a $1$-dimensional subtorus $F$, an elliptic curve.
Hence $g_0$ permutes cosets of $X/F$, and $g_0^*[F] = [F]$;
see the calculation in \cite[Lemma 2.15]{Zhang09-JDG}.
Since $g_0^*[D] = [D]$, the condition $|\langle g_0 \rangle | L| = \infty$
and the uniqueness in Lemma \ref{Oguiso} imply that $[D] = [F]$ after replacing $D$ by a multiple.
Thus, $g^*[F] = [F]$ for all $g \in G$. So $g(F)$
is also a fibre of $X \to X/F$, otherwise,
$0 < (g(F), F) = (F, F) = 0$,
a contradiction.
Hence $X \to X/F$ is a $G$-equivariant fibration.
Theorem \ref{Surface equiv} is proved.
\end{proof}

The base variety of a Lagrangian fibration is known to be a projective space
when it is smooth. This is the case in the situation of Theorem \ref{HK equiv}.

\begin{proof}[Proof of Theorem~{\rm\ref{HK equiv}}]
Set $L := \NS(X)$ as before.

For the `if' part: Take a hyperplane $H$
on $\bbP^n$, and denote its pullback $\phi^*H$ by $P$.
Then we have $g^*P = \phi^*\alpha(g)^*H$.
Since $\alpha(g)$ is an automorphism of $\bbP^n$, we have $\alpha(g)^*H\sim H$. This implies that $g^*P \sim P$.
Hence $g$ is of null entropy by Lemma \ref{e.v.}.

For the `only if' part: Suppose that $G$ is infinite and of null entropy.
By Lemma \ref{finite kernel}, $G|L$ is also infinite.
By Lemma \ref{Oguiso}, there exists a unique nonzero ray $\bbR_{>0}[D] \subseteq \partial \cC(L_\bbR) \cap \NS_{\bbR}(X)$
such that $g^*D \equiv D$ (i.e., $g^*D \sim D$ in the current case) for any $g \in G$.
Further, $q_X(D) = 0$ and $D$ can be chosen to be an integral divisor, where $q_X$ is the Beauville-Bogomolov quadratic form.
Moreover, $D$ is in the closed movable cone $\Mov(X)$ of $X$ and hence also in the closure of birational K\"ahler cone $\overline{\BK}(X)$ of $X$.

Now by \cite[Corollary 1.1]{Matsushita13},
$D$ gives rise to a rational Lagrangian fibration $\phi:X \ratmap \bbP^n$, i.e.,
there exists a birational map $\tau : X\ratmap X'$ to another hyperk\"ahler manifold $X'$ such that the linear system $|D'|$ gives rise to a Lagrangian fibration $\phi' : X' \to \bbP^n$
which is just the Stein factorization of
$$\Phi_{|D'|} : X' \to \bbP(H^0(X',D'))$$
where $D':=\tau_*D$ and $\phi=\phi'\circ \tau$.
Equivalently, we may also replace $D'$ by a primitive class such that $\Phi_{|D'|}$ itself is already a Lagrangian fibration,
using the fact that $\Pic \bbP^n \cong \bbZ$.
Replacing $(X, D)$ by $(X', D')$ and $G$ by $\tau G \tau^{-1}$,
we may assume that $\Phi_{|D|} : X \to \bbP^n$ is already a holomorphic Lagrangian fibration.
We only have to show that $G \le \Bir(X)$ descends to a regular group action on $\bbP^n$.
This is true because $g^*D \sim D$ for all $g \in G$ implies that $g$ induces an isomorphism
of $\bbP^n = \bbP(H^0(X, D))$.
This proves Theorem \ref{HK equiv}.
\end{proof}

Part (1) below is contained in \cite[Theorem 5.1]{Cantat11},
\cite[Theorem 3.2]{Cantat14} or \cite[Theorem 3.1]{Zhang08}.
For commutative group actions in higher dimension; see \cite[Theorem 1.1]{DS04}.
For an automorphism $g$ of a variety $X$, denote by $\Per(g)$ the set of
$g$-{\it periodic} points, i.e.,
the set of points $x\in X$ such that $g^s(x) = x$ for some integer $s > 0$ depending on $x$.

\begin{thm} [{cf.~\cite[Theorem 3.2]{Cantat14}, \cite[Theorem 5.1]{Cantat11}, or \cite[Theorem 3.1]{Zhang08}}] \label{Surface auto}
Let $X$ be a smooth projective surface. Assume $D$ is a numerically nonzero $\bbR$-divisor
such that $D^2 \ge 0$ (this holds when $D$ is nef). Assume further that
$g_i \in \Aut(X)$ ($i = 1, 2$)
are of positive entropy and polarized by $D$.
Then we have:
\begin{enumerate}
  \item[(1)]
  $g_1^{s_1} = g_2^{s_2}$ holds in $\Aut(X)|\NS(X)$ for some $s_i \in \bbZ \setminus \{0\}$.
  \item[(2)] Suppose that either $\Per(g_1) \cap \Per(g_2) \ne \emptyset$, or $X$ is not birational to an abelian surface.
  Then $g_1^{t_1} = g_2^{t_2}$ holds in $\Aut(X)$ for some $t_i \in \bbZ \setminus \{0\}$.
\end{enumerate}
\end{thm}

\begin{proof}
Let $L := \NS(X)/(\torsion)$ and $G=\langle g_1, g_2 \rangle$.
By assumption $[D]\in \overline \cC(L_{\bbR})\setminus\{0\}$ is a common eigenvector of $G|L$.
So by Lemma \ref{Lattice equiv}, replacing $G$ by a finite-index subgroup, we have $N(G|L)$ is a subgroup of $G|L$ and $G/N(G) \cong (G | L)/(N(G | L)) \cong \bbZ$.
Note that $N(G)$ is the inverse of $N(G | L)$ via the natural homomorphism $G \to G | L \le \OO(L)'$, and $N(G) | L = N(G | L)$ is finite by Lemma \ref{fin}.
Replacing $g_i$ by powers, we may assume that $g_1 = g_2$ (modulo $N(G)$).
Hence part (1) follows from Lemma \ref{Group}.

For part (2), since $G \ne N(G)$ and by \cite[Proposition 1]{Cantat99},
we may assume that either $X$ is a rational surface, or $X$ is a minimal surface which is either a $K3$ surface, an Enriques surface or an abelian surface.

Note that $N(G |L)$ is finite. Take an ample divisor $H'$ on $X$ and set
$$[H] := \sum_{c \in N(G | L)} c^*[H'] .$$
Then
$$N(G) \le \Aut_{[H]}(X) := \{g \in \Aut(X) : g^*[H] = [H]\} .$$
Since $|\Aut_{[H]}(X) : \Aut_0(X)| < \infty$ by \cite[Theorem 4.8]{Fujiki78} or \cite[Proposition 2.2]{Lieberman78},
we have $|N(G) : N(G) \cap \Aut_0(X)| < \infty$.
If $\Aut_0(X) = 1$ (this is true when $X$ is a $K3$ surface or an Enriques surface),
then $N(G)$ is finite and hence part (2) follows from Lemma \ref{Group}.
We may assume that $\Aut_0(X) \ne 1$.
This assumption and $G \ne N(G)$ imply that $X$ is not a rational surface (cf. Lemma \ref{irrat}).

Finally, we assume that $X$ is an abelian surface and
$\Per(g_1) \cap \Per(g_2) \ne \emptyset$.
Replacing $g_i$ by some common power, we may assume that both $g_i$ fix one common point $P \in X$.
Note that for any $s \in \bbN$, we have
$c_s := g_1^s g_2^{-s} \in N(G) \le \Aut_{[H]}(X)$.
Since $|N(G) : N(G) \cap \Aut_0(X)| < \infty$, we have
$c_s = c_t \,\, (\text{\rm modulo} \,\, N(G) \cap \Aut_0(X))$
for some $t > s$. This implies that
$$c_sc_t^{-1} \in N(G) \cap \Aut_0(X) \, \le \, \Aut_0(X).$$
Now since $\Aut_0(X)$ consists of translations and $c_sc_t^{-1}$ fixes the point $P$, we have $c_sc_t^{-1} = \id$.
So $g_1^{t-s} = g_2^{t-s}$ in $\Aut(X)$.
This proves Theorem \ref{Surface auto}.
\end{proof}

\begin{proof}[Proof of Theorem~{\rm\ref{HK auto}}]
Apply Proposition \ref{Lattice} to $L := \NS(X)$ (also need Lemma \ref{finite kernel} and Lemma \ref{Group}).
\end{proof}

\begin{remark}\label{rHK auto}
For the converse direction of Theorem \ref{HK auto}, assume that both $g_i$ in $\Bir(X)$
are of positive entropy such that
$g_1^{t_1}=g_2^{t_2}$ for some $t_i \in \bbZ \setminus \{0\}$.
Replacing $g_i$ by $g_i^{-1}$ if necessary, we may assume that both $t_i > 0$.

Since $g_1$ preserves the closed movable cone $\Mov(X)$ (see \ref{cone}) of $X$ and this cone spans $\NS_{\R}(X)$,
the generalized Perron-Frobenius theorem in \cite{Birkhoff67} (see also Lemma \ref{Bir}) implies that
$g_1^*D \equiv \lambda_1 D$ for some movable $\R$-divisor $D$ and
with $\log \lambda_1 > 0$ the entropy of $g_1$.
Thus $(g_2^{t_2})^*D \equiv \lambda D$
with $\lambda = \lambda_1^{t_1} > 1$. Now the uniqueness result in Lemma \ref{e.v.}
for $g_2^{t_2}$, implies that $g_2^*D \equiv \lambda_2 D$ with $\log \lambda_2$ the entropy
of $g_2$ and $\lambda_2^{t_2} = \lambda$.

If both $g_i$ are automorphisms, we can take $D$ to be a nef divisor.

The above argument also applies to surface automorphisms, so as to get the converse to Theorem \ref{Surface auto} (1).
\end{remark}

\begin{proof}[Proof of Proposition~{\rm\ref{PropA}}]
Denote by $L := \NS(X)/(\torsion)$ so that $G | L$ is infinite by the assumption.
Hence the first assertion is just \cite[Theorem 2.1]{Oguiso07}.
Replacing $G$ by a finite-index subgroup, we may assume that
$G | L \cong \bbZ^{\oplus s}$ for some $1 \le s \le \rho(X) - 2$.

We still have to give a better upper bound for $s$. Note that the Kodaira dimension $\kappa(X) \le 1$
since $G | L$ and hence $G$ are infinite.
If $\kappa(X) \ge 0$, we descend $G$ to a biregular action on the minimal model of $X$ and may assume that $X$ is already minimal, i.e., the canonical divisor $K_X$ is nef.
If $\kappa(X) = 0$, by the classification theory of surfaces the Picard number
$\rho(X) \le h^{1,1}(X) \le 20$
(with the equality possibly holds only when $X$ is a $K3$ surface),
so the proposition follows in this case.

If $\kappa(X) < 0$, $X$ is either a rational surface or an irrational ruled surface.

If $X$ is an irrational ruled surface, replacing $G$ by a finite-index subgroup,
we may assume that $G$ descends to a biregular action on a relatively minimal model $X_m$,
after $G$-equivariant blowdowns of $-1$-curves in the unique $\bbP^1$-fibration on $X$;
replacing $X$ by $X_m$, we may assume
that $X$ is already relatively minimal. Now this $X$ has Picard number two,
both extremal rays of the nef cone of $X$ are fixed by $G$,
contradicting the uniqueness of the $G$-stable nef ray in Lemma \ref{Oguiso}.

If $X$ is rational, we may assume that the pair $(X, G)$ is minimal
and hence $K_X^2 = 0$ by \cite[Theorem 2]{Gizatullin81}.
Thus the Picard number of $X$ is $10$ and hence $s \le 8$.
This proves the proposition.
\end{proof}


\end{document}